\documentclass[12pt]{article}

\usepackage[english]{babel}

\usepackage[letterpaper,top=2cm,bottom=2cm,left=3cm,right=3cm,marginparwidth=1.75cm]{geometry}

\usepackage{amsmath}
\usepackage{amsfonts}
\usepackage{amsthm}
\usepackage{graphicx}
\usepackage{color}
\usepackage[colorlinks=true, allcolors=blue]{hyperref}
\usepackage{comment}
\usepackage{array}  

\newtheorem{prop}{Proposition}

\newtheorem{rem}{Remark}
\newtheorem{cor}{Corollary}
\newtheorem{lem}{Lemma}

\newcommand{\tb}[1]{\textbf{#1}}
\newcommand{\Hone}{\mathcal{H}^1(\mu)}
\newcommand{\HoneStandard}{\mathcal{H}^1(0, 1)}
\newcommand{\Htwo}{\mathcal{H}^2(\mu)}
\newcommand{\Ltwo}{L^2(\mu)}
\newcommand{\Esp}{\mathbb{E}}
\newcommand{\Var}{\mathbb{V}\mathrm{ar}}

\newcommand{\prob}{\mathbb{P}}
\newcommand{\HoneNorm}[1]{\Vert #1 \Vert_{\Hone}}
\newcommand{\InfNorm}[1]{\Vert #1 \Vert_{\infty}}
\newcommand{\LtwoNorm}[1]{\Vert #1 \Vert}
\newcommand{\HoneSqNormHat}{\widehat{\HoneNorm{f}^2}}
\newcommand{\HoneDot}[1]{\langle #1 \rangle_{\Hone}}
\newcommand{\partialSum}{S_N(f)}
\newcommand{\partialSumHat}{\widehat{\partialSum}}
\newcommand{\remainder}{R_N(f)}
\newcommand{\Ustat}{\partialSumHat}

\newcommand{\eigfunell}{\varphi_\ell}

\newcommand{\hell}{h_\ell}

\newcommand{\eigvalell}{\lambda_\ell}

\newcommand{\coefell}{c_\ell}

\newcommand{\remainderfun}{\zeta}
\newcommand{\remainderfunComp}{\bar{\remainderfun}}

\newcommand{\R}{\mathbb{R}}
\newcommand{\N}{\mathbb{N}}
\newcommand{\domain}{\mathbb{X}}

\newcommand{\eigenfunInf}{\InfNorm{\eigfunell}}
\newcommand{\eigenfunDerInf}{\InfNorm{\eigfunell'}}

\newcommand{\UB}[1]{{#1}^{\textrm{UB}}}
\newcommand{\coefellUB}{c_\ell^\textrm{UB}}
\newcommand{\helltilde}{\widetilde{\hell}}
\newcommand{\helltildeUB}{\widetilde{\hell}^\textrm{UB}}
\newcommand{\deltaUB}{\UB{\delta}}

\usepackage{color}

\newcommand{\p}{\mathbb{P}}

\newcommand{\cH}{\mathcal{H}}
\newcommand{\Hnorm}[1]{\Vert #1 \Vert_{\cH}}
\newcommand{\HSqNormHat}{\widehat{\Hnorm{f}^2}}
\newcommand{\Span}{\mbox{Span}}
\newcommand{\fhat}{\widehat{f}}
\newcommand{\sh}{\textrm{sh}}
\renewcommand{\th}{\textrm{th}}
\newcommand{\ch}{\textrm{ch}}
\renewcommand{\coth}{\textrm{coth}}

\newcommand{\mb}{\mathbf}
\newcommand{\doe}{\mb{X}}

\usepackage{authblk} 

\title{Non-asymptotic confidence regions on RKHS. \ \ \ \ \  The Paley-Wiener and  standard Sobolev space cases.}

\author[1]{Fabrice Gamboa}
\author[2]{Olivier Roustant}

\affil[1]{{\small UMR CNRS 5219, Institut de Math\'ematiques de Toulouse, Universit\'e de Toulouse, France}}
\affil[2]{{\small UMR CNRS 5219, Institut de Math\'ematiques de Toulouse, INSA, Universit\'e de Toulouse, France}}

\begin{document}   

\maketitle

\begin{abstract}
   We consider the problem of constructing a global, probabilistic, and non-asymptotic confidence region for an unknown function observed on a random design. The unknown function is assumed to lie in a reproducing kernel Hilbert space (RKHS). We show that this construction can be reduced to accurately estimating the RKHS norm of the unknown function. Our analysis primarily focuses both on the Paley-Wiener (\cite{csaji2022nonparametric})  and on the standard Sobolev space settings.
\end{abstract}


\section{Introduction}

Suppose you observe the exact values of an unknown function \( f \) only at a collection of  random points \( X_1, \ldots, X_n \in \domain \). Here,  \( (X_i) \) is an i.i.d. sequence  with marginal distribution \( \mu \) supported on a compact set \( \domain \subset \mathbb{R}^d \) having non void interior.
The problem that we will address in this paper is the construction of a non asymptotic global confidence set for all the values of $f$ on $\domain$. More precisely under some assumptions of $f$, for $0<\alpha<1$, we provide a random set of functions, $\Omega(\alpha)$ built on the sample $X_1,\ldots, X_n$, such that 
\begin{equation}
\p\left(f\in\Omega(\alpha)\right)\geq 1-\alpha
\label{eq:confidence}
\end{equation}
The confidence region is constructed by using  an Hilbertian analysis with reproducing kernel Hilbert space (RKHS) tools. 
\ \\

\noindent
Recall that a RKHS $\cH$ is a Hilbert space of functions such that for any fixed  
$x \in \domain$, the (evaluation) linear form  $f \mapsto f(x)$ is continuous. We refer to \cite{paulsen2016introduction, Saitoh_RKHSbook} as general references on RKHS and to \cite{berlinet2011reproducing} as a reference to its use in statistics. Roughly speaking, many properties
on a RKHS may be performed  
using operations only on the associate semi-positive definite kernel $K(\cdot,\cdot)$ acting on $\domain\times \domain$. Notice that we will consider here only positive definite kernels. A typical example of such a 
property is associated to the generic interpolation problem (see Chapter 3 of \cite{paulsen2016introduction}). The so-called representer theorem states that the function minimizing the RKHS norm given  some prescribed values $y_1,\dots, y_n$ at $X_1,\dots, X_n$ lies in $\Span\{K(X_i,.), i=1,\dots,n\}$. Furthermore, the unique solution $\fhat$ is easily built by solving a linear system involving the Gram matrix $(K(X_i,X_j))_{1 \leq i,j \leq n}$ and the vector of observations $(y_i)_{i=1, \dots, n}$. A well-known class of RKHS is given by the Sobolev spaces $\cH^\ell(\mu)$, which consist of sets of measurable functions whose weak derivatives up to order $\ell\geq 1$ belong to $L^2(\mu)$. Here, $\mu$ is assumed to be absolutely continuous with respect to the Lebesgue measure on $\domain$. 

\ \\
\noindent
In our paper, we will focus on the case where $\cH$ is $\Hone$. The problem of constructing a set \( \Omega(\alpha) \) that satisfies (\ref{eq:confidence}) is akin to  what is now known as conformal inference. Conformal inference has a long and rich history, beginning in the early 1940s with the pioneering work of Wilks in  reliability \cite{wilks1941determination,wilks1942statistical}. The modern framework of conformal prediction was later formalized in the early 2000s by a research group centered around Vapnik and Vovk (see, e.g., \cite{vovk2005algorithmic}).  In recent years there has been a remarkable resurgence of interest in conformal prediction. Indeed, it is driven by the rapid progress of black-box models in neural machine learning. As these models become increasingly complex and opaque, the need for reliable uncertainty quantification methods becomes crucial.
Conformal prediction provides a flexible and model-agnostic framework for quantifying uncertainty predictions. It enables the construction of prediction sets that, under mild assumptions, guarantee a desired coverage probability. This makes conformal prediction a powerful tool for building reliable, finite-sample-valid prediction sets that contain the quantity of interest with high probability. 

\ \\
\noindent
Recently, the conformal prediction problem in RKHSs has been successfully addressed using a clever, kernel-tailored normalization within the classical split-conformal method
 (see \cite{pion2024gaussian, allain2025scalable}). First, note that our objective is to achieve a significantly stronger result than what is typically obtained with the classical conformal prediction method. Specifically, equation (\ref{eq:confidence}) offers multiple non-asymptotic conditional coverage guarantees. 
 More precisely, given any integer $\ell \geq 1$ and any points 
 $x_{n+1},x_{n+2},\ldots x_{n+\ell}\in\domain$,  one can readily derive from $\Omega(\alpha)$ a collection of random intervals $I_1,\ldots, I_\ell$ such that:
 \begin{equation*}
     \p\left(f(X_{n+1})\in I_1,\dots, f(X_{n+\ell})\in I_\ell \left |\right.X_{n+1}=x_{n+1},\cdots,X_{n+\ell}=x_{n+\ell} \right)\geq 1-\alpha.
 \end{equation*}
In contrast, the classical conformal prediction framework does not generally provide non-asymptotic conditional guarantees. Instead, it yields a single random interval $I$ satisfying the marginal coverage property:
\begin{equation*}
    \p(f(X_{n+1})\in I)\geq 1-\alpha. 
\end{equation*}
 Secondly, our setting would more closely align with the conformal prediction framework if we considered noisy observations. However, for the sake of simplicity, we do not address this case here. Nevertheless, our results can be readily extended to this setting, provided the noise is assumed to be bounded.
 In our paper, we address the problem of the construction of confidence regions in RKHS by focusing on the estimation of the norm of the unknown function. Indeed, as observed in \cite{csaji2022nonparametric, csaji2023improving},  this construction can be effectively addressed when one has access to both an estimator of the RKHS norm and a precise understanding of its non-asymptotic probabilistic behavior. The RKHS considered in \cite{csaji2022nonparametric, csaji2023improving} is the Paley-Wiener space, which consists of square-integrable functions with band-limited Fourier spectra. Since the RKHS norm in this case coincides with the classical \( L^2 \) norm, it can be directly estimated using a basic Monte Carlo method. 
 \ \\
 
 \noindent
In our paper, we revisit the Paley-Wiener case and mainly work in the space \( \Hone\) with \( \domain=[0,1] \).
More precisely, in the next section, we begin with a brief overview of some known properties of RKHS and explain how Equation~\eqref{eq:confidence} can be derived. We then briefly revisit the results presented in \cite{csaji2022nonparametric, csaji2023improving} for the Paley-Wiener space. In Section \ref{sect:grad}, we focus on the case where, at the sample points, both the values of the function and its derivative are observed (assuming, in particular, that the derivative exists and is defined pointwise).
 The more difficult case where the only observations at hand are the function values is considered in Section~\ref{sect:derFreeConfReg}.
Section~\ref{sect:simu} presents and discusses numerical illustrations of our results. One of the main observations is that the trust regions generally tend to overcover the target function. This phenomenon is primarily due to two factors. First, the non-asymptotic probability estimation is based on Hoeffding's inequality, which can lead to conservative intervals. Second, the local functional inequality (\ref{eq:ApproxErrorIneq}) relies on the Cauchy–Schwarz inequality, which may also contribute to the conservativeness.
In  Section~\ref{sect:concl} we summarize the main findings and discuss potential extensions to more general RKHS settings.
Some technical results are postponed to the Appendix.

\section{General framework}

\subsection{Construction of confidence regions with RKHS}
Let $\cH$ be a RKHS on  an abstract set $\domain$ with kernel $K$. Let $\mu$ a probability measure supported on $\domain$.
Now let $f \in \cH$ and let $\doe = \{x_1, \dots, x_n\} \subseteq \domain^n$ be a deterministic design. 
Given the observations $f(x_1), \dots, f(x_n)$ we consider the prediction $\hat{f}$ equal to the interpolator function in the RKHS $\cH$ with minimal norm,
\begin{equation} \label{eq:krigingMean}
   \hat{f}(x) = K(x, \doe) K(\doe, \doe)^{-1} f(\doe) \qquad (x \in \domain) 
\end{equation}
where $K(x, \doe)= (K(x, X_i))_{1 \leq i \leq n}$ is a row vector of size $n$, $K(\doe, \doe) = (K(X_i, X_j))_{1 \leq i, j \leq n} $ is a matrix of size $n$, and $f(\doe) = (f(X_i))_{1 \leq i \leq n}$ is the column vector of length $n$.\\
From now on, we assume that $\domain \subset \R^d$ and
we aim at providing a distribution-free prediction region of the form
\begin{equation} \label{eq:conformalRegion}
\prob \left( \forall x \in \domain, \, \vert f(x) - \hat{f}(x) \vert < z_\alpha s(x) \right) \geq 1 - \alpha.
\end{equation}
when $x_1, \dots, x_n$ are drawn at random independently from $\mu$. Here, $\alpha \in (0, 1)$ is a confidence level (typically $\alpha = 0.05$). To do so we will use the crucial property that in RKHS the $L^\infty$ infinite norm of a function $h$ is bounded by its RKHS norm. Indeed, recall that $\cH$ is characterized by the so-called reproducing property 
\begin{equation} \label{eq:reproducingProperty}
    h(x) = \langle h, K(x, .) \rangle_{\cH}, \qquad \forall x \in \domain, \forall h \in \cH
\end{equation}
Then, observing that $K(x, x) = \Hnorm{K(x, .)}^2$, a direct application of the Cauchy-Schwartz inequality gives
\begin{equation}
    \vert h(x) \vert \leq \Hnorm{h}  \sqrt{K(x, x)}
\end{equation}
Using the fact that $h \mapsto h(x) - \hat{h}(x)$ is linear, one can prove in a similar way the following result.

\begin{prop}[Approximation error in RKHS \cite{WendlandBook}]
\label{prop:approxErrorRKHS}
Let $x \in \domain$. Let $f \in \cH$ and $\hat{f}$ the minimal norm interpolator at $x_1, \dots, x_n$ defined in \eqref{eq:krigingMean}. Then, we have
\begin{equation} \label{eq:ApproxErrorIneq}
    \vert f(x) - \hat{f}(x) \vert \leq \Hnorm{f}  \sqrt{C(x)}
\end{equation}
where $C$ is the so-called power function, defined by
\begin{eqnarray*}
    C(x) &=& \Hnorm{K(x, .) - K(x, \doe)K(\doe, \doe)^{-1}K(\doe, .)}^2 \\
    &=& K(x,x) - K(x, \doe)K(\doe, \doe)^{-1}K(\doe, x).
\end{eqnarray*}
Furthermore, \eqref{eq:ApproxErrorIneq} is an equality if $f$ is proportional to $K(x, .) - K(x, \doe)K(\doe, \doe)^{-1}K(\doe, .)$.
\end{prop}

\begin{proof}
The proof can be found e.g. in \cite{WendlandBook}, Section 11, Th. 11.4. For the sake of a self-contained presentation, we recall here the main arguments. First, let us write 
$$ \hat{f}(x) = w(x)^\top f(\doe) $$
where $w(x)$ is the column vector defined by $w(x)^\top = K(x, \doe)K(\doe, \doe)^{-1}$. Then using the reproducing property, we have
$$ f(x) - \hat{f}(x) = \langle f, K(x, .) - w(x)^\top K(\doe, .) \rangle_{\cH}$$
Applying the Cauchy-Schwartz inequality gives,
$$  \vert f(x) - \hat{f}(x) \, \vert \leq \Hnorm{f}  \Hnorm{K(x, .) - w(x)^\top K(\doe, .)}.$$
This leads to \eqref{eq:ApproxErrorIneq} with the first expression of $C(x)$. The second expression of $C(x)$ is obtained by expanding the square norm and using the second form of the reproducing property $K(x, x') = \langle K(x, .), K(x', .) \rangle_{\cH}$, obtained from \eqref{eq:reproducingProperty} with $f = K(x', .)$:
\begin{eqnarray*}
C(x) &=&  K(x,x) -  2 w(x)^\top K(\doe, x) + w(x)^\top K(\doe, \doe) w(x) \\
&=& K(x,x) - K(x, \doe) K(\doe, \doe)^{-1} K(\doe, x)
\end{eqnarray*}
Finally, the equality case in \eqref{eq:ApproxErrorIneq} comes from the equality case in Cauchy-Schwarz inequality.
\end{proof}
 
We assume that $\doe = \{X_1, \dots, X_n\}$ is a random design where $X_1, \dots, X_n$ are independent random variables with distribution $\mu$. 
The strength of Proposition~\ref{prop:approxErrorRKHS} is that the upper bound of the approximation error ${\vert f(x) - \hat{f}(x) \vert}$  separates the effect of the function through the RKHS norm $\Hnorm{f}$ and the quality of the design $\doe$.
Therefore, to obtain a probability region of the form \eqref{eq:conformalRegion},  it suffices to show a concentration inequality of the form
\begin{equation} \label{eq:HControlHighProba}
\prob \left( \Hnorm{f}^2 < G\left(\HSqNormHat \right)  \right) > 1 - \alpha.\end{equation}
Here, $\HSqNormHat$ is an estimator of $\Hnorm{f}^2$ and $G$ is a mapping from $\R_+$ to $\R_+$. 
Indeed, using Proposition~\ref{prop:approxErrorRKHS}, this implies that 
\begin{equation} \label{eq:predictionInterval}
\prob \left( \forall x \in \domain, \, \vert f(x) - \hat{f}(x) \vert < z_\alpha \sqrt{C(x)} \right) > 1 - \alpha 
\end{equation}
where $z_\alpha = \left( G \left(\HSqNormHat \right) \right)^{1/2}$. Then, a non asymptotic confidence region at level $\alpha$ is given by, 
\begin{equation} \label{eq:conformalRegionForm}
    \Omega(\alpha):=\left\{h\in\cH:\; \forall x\in\R,  |h(x)-\widehat{f}(x)| < z_\alpha \sqrt{C(x)}\right\}
\end{equation}
Thus, our problem comes down to derive a concentration inequality of the form \eqref{eq:HControlHighProba} for the RKHS norm of $f$.\\ 

\begin{rem}[Differences with Gaussian process prediction intervals.]
A prediction interval of $\hat{f}(x)$ can be obtained in a Gaussian process (GP) framework. In this frame, one considers a deterministic design $\doe = \{x_1, \dots, x_n\} \subseteq \domain^n$. If $Y$ is a centered GP with kernel $K$, then $Y$ conditional on $Y(x_i) = f(x_i)$ ($i=1, \dots, n$) is a GP with mean $\hat{f}(x)$ and kernel
$K_c(x, x') = K(x, x') -  K(x, \doe) K(\doe, \doe)^{-1} K(\doe, x')$.
This leads to the following prediction interval at $x$, 
$$ \prob \left( \left. Y(x) \in \left[\hat{f}(x) \pm q_{\mathcal{N}(0,1)}(1-\alpha/2)  \sqrt{K_c(x, x)} \right] \right \vert Y(x_i)= f(x_i), i=1, \dots, n  \right) = 1 - \alpha. $$
Here, $q_{\mathcal{N}(0,1)}$ denotes the quantile function of the standard Gaussian distribution.
Compared to \eqref{eq:predictionInterval}, we observe some similarity, remarking that the conditional variance is equal to the power function: $K_c(x,x) = C(x)$. However, the interval is only pointwise, which does not define a region, and depends on the quantile of the Normal distribution. Furthermore, the setting is different, because the randomness comes from the GP assumption and not from the data. 
\end{rem}

\subsection{The Paley-Wiener space}
\label{sect:PaleyWiener}
One of the simplest RKHS is the so-called Paley-Wiener space (see \cite{yao1967applications} or \cite{Saitoh_RKHSbook}, pp. 3-6). Here, $\cH$ consists in the set of all square integrable functions on $\R$  whose Fourier transform is supported on $[-\eta, \eta]$. The scalar product on this space is the classical $L^2$ one. The corresponding kernel is obtained with the sinc function:
\begin{equation}
\label{eq:palwien}  
K(x,y)=\frac{\eta}{\pi} \, \frac{\sin\left(\eta (x-y)\right)}{\eta(x-y)},\;\; (x,y\in\R).
\end{equation}
The scaling factor $\eta/\pi$ is chosen such that for all $\eta>0$ the RKHS norm coincides with the $L^2$ norm.
As pointed out in \cite{csaji2022nonparametric}, for this space, if \( \mu \) is the Lebesgue measure on \([0,1]\) and \( f \) is almost entirely supported on \([0,1]\), then the squared norm of \( f \) can be accurately estimated by the empirical average  $n^{-1} \sum_{j=1}^n (f(X_j))^2$. Furthermore, assuming that an upper bound on $\|f\|_\infty$ is known, the classical Hoeffding inequality \cite{Hoeffding_1963}
can be used to derive a non-asymptotic confidence upper bound for the RKHS norm (which, in this case, corresponds to the $L^2$ norm of $f$). 
Once the RKHS norm is localized with high probability we directly make use of Proposition \ref{prop:approxErrorRKHS} and obtain the following corollary. 
\begin{cor}
\label{cor:palwie}
Assume that there exists $C_1>0$ and $\delta_0>0$  such that $\|f\|_\infty < C_1$ and $\int_{\R\setminus [0,1]}(f(x))^2dx<\delta_0$.
Then, for $0<\alpha<1$,
$$ \prob \left( \Hnorm{f} < z_\alpha \right) > 1 - \alpha $$
where $z_\alpha > 0$ is defined by
$$z_\alpha^2 =\frac{1}{n}\sum_{j=1}^n \left(f(X_i)\right)^2 + C_1\sqrt{\frac{-\log\alpha}{2n}} + \delta_0. $$
Furthermore,
$\p(f\in\Omega(\alpha))\geq 1-\alpha$, where $\Omega(\alpha)$ is given by
\begin{equation}
    \label{eq:confipalwi}
    \Omega(\alpha):=\left\{h \in\cH:\; \forall x\in\R,  |h(x)-\widehat{f}(x)|\leq z_\alpha \sqrt{ C(x)}\right\}
\end{equation}
\end{cor}

\begin{proof}
Let $\displaystyle M_n = \frac{1}{n} \sum_{i=1}^n f(X_i)^2$. Applying the Hoeffding inequality (see Appendix A) to $Y_i^2 \in\left[ 0, C_1^2 \right]$, we obtain that for all $t>0$
$$ \prob(M_n - \Esp(M_n) < -t ) < \exp \left(- \frac{2n t^2}{C_1^2} \right) $$
Choose $t = C_1\sqrt{\frac{-\log\alpha}{2n}} $ such that $\exp \left(- \frac{2n t^2}{C_1^2} \right) = \alpha$. Then, we have 
$$ \prob(\Esp(M_n) < M_n + t) > 1 - \alpha $$
Now, recall that for the Paley-Wiener space, $\Hnorm{f}^2 = \int_\R f(t)^2 dt$. Thus, by our assumptions, $ \Esp(M_n) = \int_0^1 f(t)^2dt > \Hnorm{f}^2 - \delta_0 $, which leads to
$$ \prob(\Hnorm{f}^2 < M_n + t + \delta_0) > 1 - \alpha $$
We recognize Inequality \eqref{eq:HControlHighProba} with $G(u)=u+t+\delta_0$, and the results follows from \eqref{eq:predictionInterval}.
\end{proof}

Let us point out that in \cite{csaji2022nonparametric}, inequality~(\ref{eq:ApproxErrorIneq}) is not utilized. The authors do not control the possible values of the function $f$ outside the sample points, but only the values of the RKHS interpolant with minimal norm,
through a formulation in two optimization problems (see Equation (1) in \cite{csaji2022nonparametric}). Consequently, the confidence region obtained in that reference is overly optimistic, as it is constructed solely on a finite-dimensional subspace of the RKHS.


\section{Derivative-based confidence regions in $\Hone$}
\label{sect:grad}
In this section, we assume that both function values and derivative values are observed. Notice that functions of $\Hone$ do not admit pointwise derivatives in general. Thus, we consider the subclass of functions of $\Htwo$, for which $f'$ exists pointwise (and is continuous).
Then, we can define the Monte Carlo estimator of the Sobolev norm:
\begin{equation} \label{eq:SqNormMCEwithDer}
   \HoneSqNormHat = \frac{1}{n} \sum_{i=1}^n f(X_i)^2 + \frac{1}{n} \sum_{i=1}^n f'(X_i)^2.  
\end{equation}

One may also think of modifying the definition of $\hat{f}$ such that it interpolates not only the function values but also the derivative values. However, this modification would rely on the representer theorem for derivatives, which requires a minimal regularity of the kernel $k$, typically that the cross derivative $\frac{\partial^2 K}{\partial x \partial x'}$ exists and is continuous (see e.g. \cite{Saitoh_RKHSbook}, Theorem 2.6). This is not verified for the kernel of $H^1(\mu)$  because for all $x \in \domain$ the function $y \mapsto k(x, y)$ is not differentiable  at $x$ \cite{JAT_Poincare_quadrature}.

\begin{prop}[Conformal prediction interval when derivatives are available]
\label{prop:ConfRegionWithDer}
Let $X_1, \dots, X_n$ be i.i.d. random variables with probability distribution $\mu$. Let $f \in \Htwo$ verifying $ f(X_1)^2 + f'(X_1)^2 \leq b $ almost surely. 
Then for all $\alpha \in (0, 1)$,
$$ \prob \left( \HoneNorm{f}^2\leq \HoneSqNormHat + t_{n, \alpha} \right) \geq 1 - \alpha $$
with $t_{n, \alpha} = b \sqrt{\frac{- \ln \alpha}{2n}}$. Furthermore,  
$$
\prob \left( \forall x \in \domain, \, \vert f(x) - \hat{f}(x) \vert < z_\alpha \sqrt{C(x)}  \right) \geq 1 - \alpha. 
$$
with $z_\alpha = \left(\HoneSqNormHat + t_{n, \alpha}\right)^{1/2}$.
\end{prop}

\begin{proof}
Let $Z_i = f(X_i)^2 + f'(X_i)^2$. The random variables $(Z_i)$ are i.i.d., and lie in $[0, b]$ almost surely. Furthermore, $\Esp(Z_i) = \HoneNorm{f}^2$. Then, by Hoeffding inequality, we deduce that for all $t > 0$,
$$ \prob \left( \HoneSqNormHat - \HoneNorm{f}^2 \leq - t \right) \leq e^{- \frac{2nt^2}{b^2}} $$
Choosing $t = t_{n, \alpha}$, this gives
$ \prob \left( \HoneSqNormHat + t \leq \HoneNorm{f}^2 \right) \leq \alpha $, 
which leads to the first inequality:
$$ \prob \left( \HoneNorm{f}^2\leq \HoneSqNormHat + t_{n, \alpha} \right) \geq 1 - \alpha $$
The prediction interval follows from \eqref{eq:HControlHighProba} and \eqref{eq:predictionInterval} with $G(x) = x + t_{n, \alpha}$.
\end{proof}

\section{Derivative-free confidence regions in $\HoneStandard$}
\label{sect:derFreeConfReg}

In the whole section, we assume that $\mu$ is the uniform distribution on $[0, 1]$.

\subsection{First ideas of estimators}
When derivatives are not observed, the simplest idea is to estimate them by finite differences. Denote by $X_{(1)}, \dots, X_{(n)}$ the sample obtained from $X_1, \dots, X_n$ by sorting in ascending order. As $X_1, \dots, X_n$ are uniformly distributed, we may replace $X_{(i)} - X_{(i-1)}$ by its expectation $\approx 1/n$. Thus, 
$$ f'(X_i) \approx \frac{f(X_{(i)})-f(X_{(i-1)})}{1/n} $$
and we obtain the finite-difference version of the estimator \eqref{eq:SqNormMCEwithDer}:
\begin{equation}
   \label{eq:SqNormMCEwithoutDer}
   T_{0, n} = \frac{1}{n} \sum_{i=1}^n f(X_i)^2 + n \sum_{i=2}^n \left[ f(X_{(i)}) - f(X_{(i-1)}) \right]^2
\end{equation}

A second idea is to use the general expression of the norm in a RKHS $\cH$ with kernel $K$. Let $x_1, \dots, x_n \in [0, 1]$, and denote by $f(\doe) = (f(x_1), \dots, f(x_n))^\top$ the observations and $K(\doe, \doe) =(K(x_i, x_j))_{1 \leq i,j \leq n}$ the Gram matrix associated to $K$. Denote by $\Pi_X(f)$ the projection of $f$ onto $\text{span}(K(x_1, .), \dots, K(x_n, .))$. Then we have (see for example \cite{paulsen2016introduction}),
\begin{equation} \label{eq:SqNormGram}
	\Vert \Pi_X(f) \Vert_{\cH} ^2= f(\doe)^\top K(\doe, \doe)^{-1} f(\doe)
\end{equation}
Notice that 
$$ \Vert f \Vert_{\cH} ^2 = \Vert \Pi_X(f) \Vert_{\cH} ^2 + \Vert f-\Pi_X(f) \Vert_{\cH} ^2 $$
Thus, the supremum of $\Vert \Pi_X(f) \Vert_{\cH} ^2$ with respect to all design points $x_1, \dots, x_n$ exists and gives an approximation by below of $\Vert f \Vert_{\cH} ^2$
\begin{equation} \label{eq:SqNormGramApprox}
	\Vert f \Vert_{\cH} ^2 \underset{\geq}{\approx} \sup_{0 \leq x_1 \leq \dots \leq x_n \leq 1} \Vert \Pi_X(f) \Vert_{\cH} ^2.
\end{equation}
In the case of $\Hone$, we can go a step further by computing explicitly $\Vert \Pi_X(f) \Vert_{\cH} ^2$. Indeed, the inverse of the Gram matrix is an explicit tridiagonal matrix.
More precisely, if $0 \leq x_1 < \dots < x_n \leq 1$, it is shown in \cite{duc1973approximation} that $$ K(\doe, \doe)^{-1} = 
\begin{pmatrix}
	\alpha_1 & \beta_1 & &  & 0 \\
	\beta_1 & \alpha_2 & \beta_2 & \\
	 & \ddots & \ddots & \ddots & & \\
	 & & \beta_{n-2} & \alpha_{n-1} & \beta_{n-1} \\
	0 & &  & \beta_{n-1} & \alpha_n \\
\end{pmatrix}$$
where $\alpha_i, \beta_i$ are explicit expressions depending on $K$. For the uniform distribution on $[0, 1]$, we have $K(x,y) = a(\min(x,y)) b(\max(x,y))$ with $a(x) = \frac{1}{\sh(1)} \ch(x)$ and $b(x) = \ch(1-x)$. Denote $a_i = a(x_i), b_i = b(x_i)$, $\mu_{i,j} = a_i b_j - a_j b_i$. Then, we have
$$ \alpha_1 = \frac{a_2}{a_1} \frac{1}{\mu_{2,1}}, \quad 
\alpha_i = \frac{\mu_{i+1, i-1}}{\mu_{i, i-1} \mu_{i+1, i}} \quad (2 \leq i \leq n-1),
\quad
\alpha_n =  \frac{b_{n-1}}{b_n} \frac{1}{\mu_{n, n-1}}
$$ and $$
\beta_i = \frac{1}{\mu_{i+1, i}} \quad (1 \leq i \leq n-1)$$
This gives here
\begin{eqnarray*}
\alpha_1 &=& \th(x_1) + \coth(x_2 - x_1),   \\
\alpha_i &=& \coth(x_{i+1} - x_i ) + \coth(x_i - x_{i-1}), \quad (2 \leq i \leq n-1) \\
\alpha_n &=& \coth(x_n - x_{n-1}) + \th(1 - x_n) \\
\beta_i &=& - \frac{1}{\sh(x_{i+1} - x_i)}, \quad (1 \leq i \leq n-1)
\end{eqnarray*}
Define $x_0$ and $x_{n+1}$ such that $\frac{x_1 - x_0}{2} = x_1$ and $\frac{x_{n+1} - x_n}{2} = 1 - x_n$, that is $x_0 = - x_1$ and $x_{n+1} = 2 - x_n$. Then, after some algebra, we obtain, still with $x_1 < \dots < x_n $,
\begin{equation} \label{eq:SqNormGramExplicit}
\begin{split}
    \Vert \Pi_X(f) \Vert_{\cH} ^2 
=  \sum_{i=1}^{n} 
\left[ 
\th \left( \frac{x_{i+1} - x_i}{2} \right) \right.
&+ \left. \th \left( \frac{x_{i} - x_{i-1}}{2} \right)
\right] f(x_i)^2 \\
&+ \sum_{i=2}^{n} 
\frac{1}{\sh(x_i - x_{i-1})}  
\left[f(x_i) - f(x_{i-1}) \right]^2 
\end{split}
\end{equation}
Now considering a sample $\doe = \{X_1, \dots, X_n \}$ drawn from the uniform distribution on $[0, 1]$, we can propose two estimators of $\HoneNorm{f}^2$, obtained by replacing $x_{(i)}$ by $X_{(i)}$ either directly in \eqref{eq:SqNormGram} or in its explicit expression \eqref{eq:SqNormGramExplicit}. Notice that the estimator \eqref{eq:SqNormMCEwithDer} may appear as a limit case of \eqref{eq:SqNormGramExplicit}: 
if $x_{(i)} \approx \frac{i-1/2}{n}$ then, using that $\th(x) \underset{x \to 0}{\sim} x$ and $\sh(x) \underset{x \to 0}{\sim} x$, we obtain that 
\begin{equation*}     \Vert \Pi_X(f) \Vert_{\cH} ^2 
\approx \frac{1}{n} \sum_{i=1}^{n} 
f(x_i)^2 
+ n \sum_{i=2}^{n} 
\left[f(x_{(i)}) - f(x_{(i-1)}) \right]^2. 
\end{equation*}


To conclude this presentation of the natural estimators of $\HoneNorm{f}^2$ when derivatives are not observed, we can see that all the estimators \eqref{eq:SqNormMCEwithoutDer} - \eqref{eq:SqNormGramExplicit} can be hardly used for deriving concentration inequalities, in a non-asymptotic setting. 
This motivates the introduction of an alternative approach, based on a spectral method.

\subsection{A spectral estimator}
The construction of the spectral estimator is based on the so-called Poincaré basis, that we present now.
We denote $\langle ., . \rangle$ the dot product in $\Ltwo$ and $\Vert . \Vert$ the associated norm.
The Poincaré basis $(\eigfunell)_{\ell \in \N}$ is an orthonormal basis of $\Ltwo$ formed by eigenfunctions of the Laplacian operator with Neumann boundary conditions.
More precisely, the spectral problem of finding $\lambda \in \R$ and $f \in \Htwo$ with $\Vert f \Vert = 1$ and $f(0)>0$ such that\footnote{The conditions $\Vert f \Vert = 1$ and $f(0)>0$ are added to define the basis in a unique way, as eigenfunctions are defined up to a multiplicative constant.}
$$ f'' = - \lambda f, \qquad f'(0) = f'(1) = 0$$
has a countable set of solutions $(\eigvalell, \eigfunell)_{\ell \in \N}$
given explicitly by
\begin{equation} \label{eq:PoincareBasisExpression}
 \eigvalell = \pi^2 \ell^2, \qquad 
\eigfunell(x) = 
\begin{cases}
1 & \ell = 0 \\
\sqrt{2} \cos(\pi \ell x) & \ell \geq 1 
\end{cases}
\end{equation}
It is associated to Poincaré inequalities, and we refer to \cite{roustant2017poincare} for more details. It is also related to Sturm-Liouville theory, and called modified Fourier series in \cite{adcock2009univariate}.
It is characterized by the formula
\begin{equation} \label{eq:PoincCaracHone}
\langle f', \eigfunell' \rangle = \eigvalell \langle f, \eigfunell \rangle \quad \forall f \in \Hone
\end{equation}
By choosing $f = \varphi_m$ in \eqref{eq:PoincCaracHone} (with $m \in \N$), this property implies that the derivative of the Poincaré basis is an orthogonal system. Furthermore, one can prove that it is an orthogonal basis, which is actually a characterization among orthonormal bases of $\Ltwo$, with $\varphi_0 \equiv 1$, that belong to $\Hone$ \cite{luthen2023global}. Finally, the Poincaré basis is an orthogonal basis of $\Hone$, and we have
\begin{equation} \label{eq:eigenfunHoneNorm}
    \HoneNorm{\eigfunell}^2 = 1 + \eigvalell
\end{equation}
Now, applying Parseval inequality with the orthonormal basis $(1+\eigvalell)^{-1/2}\eigfunell$, we have
$$\HoneNorm{f}^2 
= \sum_{n \geq 0} 
\HoneDot{f, (1+\eigvalell)^{-1/2} \eigfunell} ^2
= \sum_{n \geq 0} (1+\eigvalell)^{-1} \HoneDot{f, \eigfunell}^2$$
Finally, using the specific property of the Poincaré basis \eqref{eq:PoincCaracHone}, it holds that
$$ \HoneDot{f, \eigfunell} = (1+ \eigvalell) \langle f, \eigfunell \rangle $$
This leads to
\begin{equation} \label{eq:HoneNorm}
\HoneNorm{f}^2 
= \sum_{n \geq 0} (1+\eigvalell) \coefell^2
\end{equation}
where $\coefell = \langle f, \eigfunell \rangle$ has been defined in \eqref{eq:coefell}.\\ 

To define an estimator of $\HoneNorm{f}^2$, we fix some integer $N \geq 1$ and split the sum in two parts corresponding to the partial sum and remainder of the series:
\begin{equation} \label{eq:PartialSum_Plus_Remainder}
 \Vert f \Vert^2_{\Hone}  = \partialSum + \remainder
\end{equation}
with:
\begin{eqnarray} 
   \partialSum &=& \sum_{\ell = 0}^N  (1+\eigvalell) \coefell^2 \label{eq:partialSumDef} \\
   \remainder &=& \sum_{n \geq N+1} (1+\eigvalell) \coefell^2 \label{eq:RemainderDef}
\end{eqnarray}

The first term can be estimated with a U-statistics. The remainder term is a bias term, and will be controlled by adding an assumption on $f$.

\paragraph{A $U$-statistics for estimating $\partialSum$.}
Notice that $\coefell = \Esp_\mu [f \eigfunell]$.
Thus, an unbiased estimator of $\coefell^2$ is given by the $U$-statistics \cite{Hoeffding_1963}:
$$ T_\ell = \frac{2}{n(n-1)} \sum_{1 \leq j < j' \leq n} f(X_j) \eigfunell(X_j) f(X_{j'}) \eigfunell(X_{j'}) $$
This leads to the estimator of $\partialSum$
$$\partialSumHat = \sum_{\ell=0}^N (1+\eigvalell) T_\ell.$$
Notice further that $\partialSumHat$ is also a $U$-statistics, as a linear combination of the $U$-statistics $T_\ell$. Its expression as a finite spectral expansion will be useful in the sequel: 
\begin{equation} \label{eq:spectralExp_Ustat}
   \partialSumHat 
   = \frac{2}{n(n-1)} \sum_{\ell=0}^N (1+\eigvalell) \sum_{1 \leq j < j' \leq n} \hell(X_j) \hell(X_{j'})
\end{equation}
where $\hell(x) = f(x) \eigfunell(x)$ has been defined in \eqref{eq:hell}.

\paragraph{Asssumption on $\remainder$.}
The quantity $\remainder$ is the remainder of the series $\HoneNorm{f}^2$ (see Equation (\ref{eq:HoneNorm})). We now make an assumption on the relative rate of convergence to zero of this remainder. Thus, we assume that there exists a (known) function $\remainderfun: \R^+ \to [0, 1[$, decreasing and tending to $0$ at infinity, such that for all $N \geq 0,$
\begin{equation} \label{eq:remainderAssumption}
    \remainder \leq \HoneNorm{f}^2 \remainderfun(N)
\end{equation}
We will also denote, for convenience,
\begin{equation} \label{eq:remainderfunComp}
    \remainderfunComp_N = 1 - \remainderfun(N)
\end{equation}
With this notation, \eqref{eq:remainderAssumption} is equivalent to 
\begin{equation} \label{eq:remainderAssumptionBis}
   \partialSum \geq \remainderfunComp_N \, \HoneNorm{f}^2 
\end{equation}

\subsection{Main results}

\subsubsection{Agnostic method}
Here, we assume that only information on the the bias term magnitude is known.

\begin{prop} \label{prop:IneqAgnostic}
Let $\alpha \in (0, 1)$ and $c = 2K_\infty \sum_{\ell = 0}^N (1+\eigvalell) \InfNorm{\eigfunell}^2$. Let $n \in \N$ such that $2[n/2] \geq -\log(\alpha)c^2/\remainderfunComp_N^2$. Let $t>0$ such that
\begin{equation} \label{ineq:tAgnostic}
\left( \frac{- \log(\alpha)}{2[n/2]} \right)^{1/2} c < t < \remainderfunComp_N.
\end{equation}
Then, we have,
$$ \prob \left(
\HoneNorm{f} \leq \left(
\frac{\max \left( \Ustat, 0 \right)  }{\remainderfunComp_N - t} 
 \right)^{1/2} 
\right) \geq 1 - \alpha .$$
\end{prop}
\bigskip
\begin{proof}
Let us consider the function 
$$ g(x,y) = \sum_{\ell = 0}^N (1+\eigvalell) \hell(x) \hell(y). $$
Notice that $\displaystyle \Ustat = \frac{2}{n(n-1)} \sum_{j < j'} g(X_j, X_j')$.\\
Using that $f(x) \leq  \HoneNorm{f} \sqrt{K_\infty}$, we have the bound for $g$,
$$ \vert g \vert \leq \frac{c}{2} \HoneNorm{f}^2.$$
By applying the Hoeffding's inequality for one-sample U-statistics \cite[\S 5a]{Hoeffding_1963} (see Appendix A), we have that for all $t' > 0$,
$$ \prob \left( \Ustat - \partialSum < -t' \right) < \exp \left( - 2k \left(\frac{t'}{c \HoneNorm{f}^2}\right)^2 \right) $$
with $k = [n/2]$. Setting $t = t'/\HoneNorm{f}^2$, we get,
$$ \prob \left( \Ustat < \partialSum -t \HoneNorm{f}^2 \right) < \exp \left( -  \frac{2k t^2}{c^2} \right).$$
By assumption on the bias, $\remainderfunComp_N \HoneNorm{f}^2 \leq \partialSum$. This implies that, 
$$ \{ \max(\Ustat, 0) < \remainderfunComp_N \HoneNorm{f}^2 -t \HoneNorm{f}^2 \} \subseteq \{ \Ustat < \partialSum -t \HoneNorm{f}^2 \}, $$
hence
$$ \prob \left( \max(\Ustat, 0) < (\remainderfunComp_N - t) \HoneNorm{f}^2 \right) < \exp \left( -  \frac{2k t^2}{c^2} \right).$$
The result follows by remarking that the condition \eqref{ineq:tAgnostic} on $t$ is equivalent to require $ \remainderfunComp_N - t > 0$ and $\exp \left( -  \frac{2k t^2}{c^2} \right) \leq \alpha$. 
\end{proof}

This inequality is attractive because it requires only an assumption on the bias. However, it is only valid for large values of $n$, as \eqref{ineq:tAgnostic} implies that $n \geq -\log(\alpha)c^2/\remainderfunComp_N^2 $ and thus $n \geq -\log(\alpha)c^2$. For the Uniform distribution and $\alpha = 0.05$, knowing that $K_\infty = 1/\sh(1) \approx 0.85$, we obtain $n > 34.7 \left(\sum_{\ell = 0}^N (1+\ell^2 \pi^2)\right)^2$. For $N= 1, 2$, this gives $n> 4\,100$ and $n> 91\,505$ respectively. This is due to the divergence of the series $\sum_\ell (1+\eigvalell)$ as $\eigvalell = O(\ell^2)$. To improve this result, one must use an information on the decreasing rate for the coefficients $\coefell$, linked to the regularity of the function.

\subsubsection{Assuming regularity on $f$.}
\label{sect:derFreeConfRegCoefDecrease}
To obtain realistic bounds on $\HoneNorm{f}$, we need to involve the coefficients $\coefell$ of the expansion of $f$ on the eigenbasis $\eigfunell$. As explained in Appendix B (paragraph ``Bounding $\coefell$''), they decrease at least at the order $\eigvalell^{-1/2}=O(\ell^{-1})$ and their order of decreasing is linked to the regularity of $f$. Here, we will need an order of  $\eigvalell^{-p}$ for some $p>1$, which is achieved for functions $f$ that are slightly more regular than those of $\Htwo$. Thus we assume: 
\begin{equation} \label{eq:coefSpeedHyp}
   \exists A >0, \exists p > 1, \forall \ell \in \N^\star, \quad \vert \coefell \vert \leq A \, \eigvalell^{-p}    
\end{equation}
Using the expansion $f(x) = \sum_{\ell} \coefell \eigfunell(x)$, such assumption implies
\begin{eqnarray}
\InfNorm{f} &\leq& A \sqrt{2}\, \zeta(2p)/\pi^{2p} \label{eq:fInfUBreg}\\
\InfNorm{f'} &\leq& A \sqrt{2} \,\zeta(2p-1)/\pi^{2p-1} \label{eq:fDerInfUBreg}
\end{eqnarray} 
where $\zeta$ is the Riemann zeta function $\zeta(p) = \sum_{\ell \geq 1} \ell^{-p}$. Notice that these two upper bounds are finite for $p>1$, which justifies this condition on $p$ for the computations.\\

In order to involve $\coefell$ in the concentration inequalities, we will use a decomposition derived from the Sobol-Hoeffding decomposition of U-statistics for estimating a sum of squares. A general approach to derive concentration inequalities for second-order U-statistics is presented in \cite{houdre_ReynaudBouret}. 
However, remarkably for our statistics, although of second-order, a single Hoeffding's type inequality provides a concentration inequality quantifying the deviation by below. This is stated in the next lemma.

\begin{lem}[Deviation by below for second-order U-statistics estimating sum of squares] \label{lem:devBelowSumSqUstat}
Let $X_1, \dots, X_n$ be i.i.d. random variables. For $N \in \N$, let $g_0, \dots, g_N$ be functions from $\R$ to $\R$ such that $g_\ell(X_j)$ are almost surely bounded ($\ell=0, \dots, N, j=1, \dots, n$). Let $\theta = \sum_{\ell=0}^N  \Esp(g_\ell(X_1)^2)$ and the let consider the associated U-statistics 
$$ U_{n} = \frac{2}{n(n-1)}
\sum_{\ell = 0}^N  \sum_{1 \leq j < j' \leq n}
g_\ell(X_j)g_\ell(X_{j'}).$$
Denote $\mu_\ell = \Esp(g_\ell(X_1))$. Define, for any positive integer $j$,
\begin{equation} \label{eq:W_jDef}
W_j = 2 \sum_{\ell = 0}^N  \left(\mu_\ell(g_\ell(X_j) - \mu_\ell) - \frac{(g_\ell(X_j) - \mu_\ell)^2}{n-1}  \right).
\end{equation}
Let $\alpha \in (0, 1)$. Let $a_{W}, b_{W}$ be two real numbers such that $a_{W} \leq W_j \leq b_{W}$ almost surely, and $\delta_W = \vert \Esp(W_j) \vert$.
Then, for all $t > 0$, we have,
$$ \prob \left( U_n - \theta  < - (t + \delta_W) \right) \leq \exp \left( - \frac{2nt^2}{(b_{W} - a_{W})^2} \right).$$
\end{lem}

\begin{proof}
Let $\theta_\ell := \Esp(g_\ell(X_1))$
and $ U_{n, \ell} = \frac{2}{n(n-1)}
 \sum_{1 \leq j < j' \leq n}
g_\ell(X_j)g_\ell(X_{j'}) $
be the associated U-statistics. 
By writing $g_\ell(X_j) = (g_\ell(X_j) - \mu_j) + \mu_j$ we obtain the decomposition
\begin{equation} \label{eq:SobolDecUstat}
    U_{n, \ell} = \theta_\ell + \underset{\text{first-order part}}{\underbrace{\frac{2}{n} \sum_{j=1}^n \mu_\ell(g_\ell(X_j) - \mu_\ell)}} 
    + \underset{\text{second-order part}}{\underbrace{\frac{2}{n(n-1)} \sum_{1 \leq j < j' \leq n} (g_\ell(X_j) - \mu_\ell)(g_\ell(X_{j'}) - \mu_\ell)}}.
\end{equation}
This splits the statistics into its mean, a first-order part gathering contributions brought by a single variable, and a second-order part collecting the second-order interactions. Actually, one can show that \eqref{eq:SobolDecUstat} is nothing more than the so-called Sobol'-Hoeffding decomposition of $U_{n,\ell}$. Nevertheless,  we do not need investigating further this aspect.
Rather, we reorganize  \eqref{eq:SobolDecUstat} by adding the diagonal terms missing in the second-order part, and by removing them to the first-order part:
\begin{equation*} 
\begin{split}
    U_{n, \ell} &= \theta_\ell^2
     + \frac{2}{n} \sum_{j=1}^n \left(\mu_\ell(g_\ell(X_j) - \mu_\ell) - \frac{(g_\ell(X_j) - \mu_\ell)^2}{n-1}  \right) \\
     & \hspace{7cm} +  \frac{2}{n(n-1)} \left[\sum_{j=1}^n (g_\ell(X_j) - \mu_\ell)\right]^2
\end{split}
\end{equation*}
As we can see, the modified second-order part is now a square, and thus non-negative.
Summing with respect to $\ell$, we obtain that
\begin{equation} \label{eq:SobolDecUstatModified}
    U_n = \theta + \frac{1}{n} \sum_{j=1}^n W_j + \Delta
\end{equation}
where $W_j$ is defined in \eqref{eq:W_jDef} and $\Delta \geq 0$ almost surely. 
Notice that $W_1, \dots, W_n$ are i.i.d. random variables, almost surely bounded, with a negative expectation: 
$$\Esp(W_j) = -  \frac{2}{n-1}\sum_{\ell = 0}^N \Var(g_\ell(X_j)).$$
Thus, $\Esp(W_j) = - \delta_W$ (for all $j=1, \dots, d$). Then, the Hoeffding's inequality \cite{Hoeffding_1963} tells us that for all $t > 0$,
$$ \prob \left( \frac{1}{n} \sum_{j=1}^n W_j + \delta_W < - t \right) \leq \exp \left( - \frac{2nt^2}{(b_W - a_W)^2} \right). $$
Denote $\alpha_W := \exp \left( - \frac{2nt^2}{(b_W - a_W)^2} \right)$.
Using \eqref{eq:SobolDecUstatModified}, this is equivalent to, 
$$ \prob \left( U_n - \theta - \Delta  < - (t + \delta_W) \right) \leq \alpha_W. $$
As $\Delta$ is non-negative, we have, 
$$    \left\lbrace U_n - \theta < - (t + \delta_W) \right\rbrace 
\subseteq
\left\lbrace U_n - \theta - \Delta < - (t + \delta_W)  \right\rbrace,
$$
and thus
$$ \prob \left( U_n - \theta  < - (t + \delta_W) \right) \leq \alpha_W.$$
\end{proof}

This concentration inequality is of Hoeffding's type although the U-statistic is second-order. Neglecting the second-order term has been possible at the price of adding a positive term $\deltaUB$ to $t$ in the upper bound. Nevertheless, $\deltaUB$ decreases as $O(n^{-1})$ when $n$ goes to infinity, which is smaller than the order of decrease of $t$ equal to $O(n^{-1/2})$.\\
We now derive a concentration inequality to our context.
It involves upper bounds of the Poincaré basis coefficients, denoted $\coefellUB$, and of $\helltilde := f \eigfunell - \coefell$, denoted $\helltildeUB$. The computation of these upper bounds are derived in the appendix. 

\bigskip
\begin{prop} 
\label{prop:derFreeRegion}
Let $\alpha \in (0, 1)$ and $n \in \N^\star$. Define, 
\begin{eqnarray*}
s(\InfNorm{f}, \InfNorm{f'}) &=& 4 \sum_{\ell = 0}^N (1+\eigvalell) \coefellUB \helltildeUB
+ \frac{2}{n-1} \sum_{\ell = 0}^N (1+\eigvalell) \left( \helltildeUB \right)^2,\\
\delta(\InfNorm{f}) &=& \frac{2\InfNorm{f}^2}{n-1} \sum_{\ell = 0}^N (1+\eigvalell),
\end{eqnarray*}
where $\coefellUB = A \eigvalell^{-p}$ is given by \eqref{eq:coefSpeedHyp}, and $\helltildeUB$ is obtained from \eqref{eq:coefSpeedHyp},  \eqref{eq:helltildeUB} in function of 
$\InfNorm{f}$ and $\InfNorm{f'}$. 
Remarking that $s$ is an increasing function of its arguments, let $\UB{s}$ be the upper bound of $s(\InfNorm{f}, \InfNorm{f'})$ obtained by replacing $\InfNorm{f}$ and $\InfNorm{f'}$ by their upper bounds given in \eqref{eq:fInfUBreg} and \eqref{eq:fDerInfUBreg}. Define $\UB{\delta}$ similarly. Then, for all $t>0$ such that
\begin{equation*} 
t > \left( \frac{- \log(\alpha)}{2n} \right)^{1/2} \UB{s},
\end{equation*}
we have,
$$ \prob \left(
\HoneNorm{f} \leq \left(
\frac{\max \left( \Ustat + t + \deltaUB, 0 \right)  }{\remainderfunComp_N} 
 \right)^{1/2} 
\right) \geq 1 - \alpha.$$
\end{prop}
\bigskip
\begin{proof} Let us apply Lemma \ref{lem:devBelowSumSqUstat}, with $g_\ell = \sqrt{1+\eigvalell} \hell$. We have $\mu_\ell = \sqrt{1+\eigvalell} \coefell$ , $U_n = \Ustat$, $\theta = \partialSum$, and, 
$$ 
W_j = 2 \sum_{\ell = 0}^N 
(1 + \eigvalell)
\left(\coefell(\hell(X_j) - \coefell) - \frac{(\hell(X_j) - \coefell)^2}{n-1}  \right).
$$
Recall that $a_W, b_W$ are real numbers such that $a_W \leq W \leq b_W$ and $\delta_W = \vert \Esp(W_j) \vert$.
Then, we get that for all $t > 0$,
$$   \prob \left( \Ustat - \partialSum  < - (t + \delta_W) \right) \leq \exp \left( - \frac{2nt^2}{(b_{W} - a_{W})^2} \right).
$$
Now, using the inequalities \eqref{eq:helltildeUB} and \eqref{eq:varHellUB}, we can see that $\deltaUB$ and $\UB{s}$ are upper bounds of $\delta_W$ and $b_W - a_W$ respectively. 
This implies that  
\begin{equation*}
    \begin{split}
    & \prob \left( \Ustat - \partialSum  < - (t + \deltaUB) \right) \\
    & \hspace{1cm} 
    \leq 
    \prob \left( \Ustat - \partialSum  < - (t + \delta_W) \right)
    \leq \exp \left( - \frac{2nt^2}{(b_{W} - a_{W})^2} \right)
    \leq 
    \exp \left( - \frac{2nt^2}{(\UB{s})^2} \right).
    \end{split}
\end{equation*}
In particular, for all  
$ t > \left( \frac{- \log(\alpha)}{2n} \right)^{1/2} \UB{s}
$, we have 
$$   \prob \left( \Ustat - \partialSum  < - (t + \deltaUB) \right) \leq \alpha. $$
Now, by assumption on the bias, $\remainderfunComp_N \HoneNorm{f}^2 \leq \partialSum$. This implies that $$ \{ \max(\Ustat, 0) < \remainderfunComp_N \HoneNorm{f}^2 - (t + \deltaUB) \} \subseteq \{ \Ustat < \partialSum - (t + \deltaUB)  \}.$$
Finally
$$ \prob \left( \max(\Ustat, 0) < \remainderfunComp_N \HoneNorm{f}^2 - (t + \deltaUB) \right) \leq \alpha, $$
and the result follows.
\end{proof}

\section{Numerical experiments}
\label{sect:simu}
This section complements the theoretical findings of the previous sections with numerical experiments. The implementation was carried out using the R software \cite{R_software}
and the corresponding code is publicly available at \href{https://github.com/roustant}{https://github.com/roustant}.

\subsection{Confidence regions in the Paley-Wiener RKHS}
We consider the Paley-Wiener RKHS, in the setting of Section~\ref{sect:PaleyWiener}. The setup of the numerical experiments is comparable to the one of \cite{csaji2022nonparametric}.
Thus, we choose $\eta = 30$ and define the unknown function $f \in \cH$ as
\begin{equation} \label{eq:PWtestFun}
 f(x) = \sum_{m=1}^M w_m K(x, z_m)
\end{equation}
where $z_i = \frac{i-1/2}{M}$ ($i=1, \dots, M)$ form a regular sequence of $[0, 1]$ and $w_1, \dots, w_M$ are drawn independently from the uniform distribution on $[-0.1, 0.1]$. We fix $M=20$, a value large enough to capture a range of patterns in the behavior of $f$.
Then, we set $n=10$, and sample $X_1, \dots, X_n$ independently from the uniform distribution on $[0, 1]$. 

We first compute the confidence interval on the RKHS norm of $f$, denoted $I_\alpha(\doe)$, as presented in Corolllary~\ref{cor:palwie} (where $\doe = \{X_1, \dots, X_n \}$).
Thus $I_\alpha(\doe) = [0, z_\alpha]$.
Here, $C_1$ and $\delta_0$ have been computed numerically based on the expression of $f$. 
In practice, however, these values should be bounded by above based on prior knowledge.
As an initial indication of the usefulness of this confidence interval, we verify in our numerical experiments that $z_\alpha < (\delta_0 + C_1^2)^{1/2}$. This is expected as the right hand side is an obvious upper bound of $\Hnorm{f}$:
$$ \Hnorm{f}^2 = \LtwoNorm{f}^2 = \delta_0 + \int_0^1 f^2(x)dx \leq \delta_0 + C_1^2$$
Now, we can further assess the actual coverage of $I_\alpha(\doe)$. 
Indeed, in this example, the RKHS norm of $f$ is given explicitly by
$$ \Hnorm{f}^2 = w^\top K(\mb{Z}, \mb{Z}) w$$
where $w = (w_1, \dots, w_M)^\top$ and $K(\mb{Z}, \mb{Z}) = (K(z_m, z_{m'}))_{1 \leq, m, m' \leq M}$. Figure~\ref{fig:PaleyWienerNormCoverage} depicts a barplot of $I_\alpha(\doe)$ for $200$ independent samples $(X_1, \dots, X_n)$ for $\alpha = 0.25$. We can see that the intervals are slightly too conservative. The empirical coverage, defined as the proportion of these intervals that contain the true value $\Hnorm{f}$, is approximately equal to $0.98$, instead of the minimal expected value $1 - \alpha = 0.75$. A similar result is observed for other values of $\alpha$. For instance, for $\alpha = 0.5$, the empirical coverage is approximately $0.86$. As $C_1$ and $\delta_0$ have been fixed at their minimal possible values, this means that the overestimation comes from the Hoeffding's inequality used to compute $I_\alpha(\doe)$. 

\begin{figure}[h!]
    \centering
    \includegraphics[width=0.8\linewidth, height=6cm, trim=1cm 1.5cm 1cm 1.5cm, clip = TRUE]{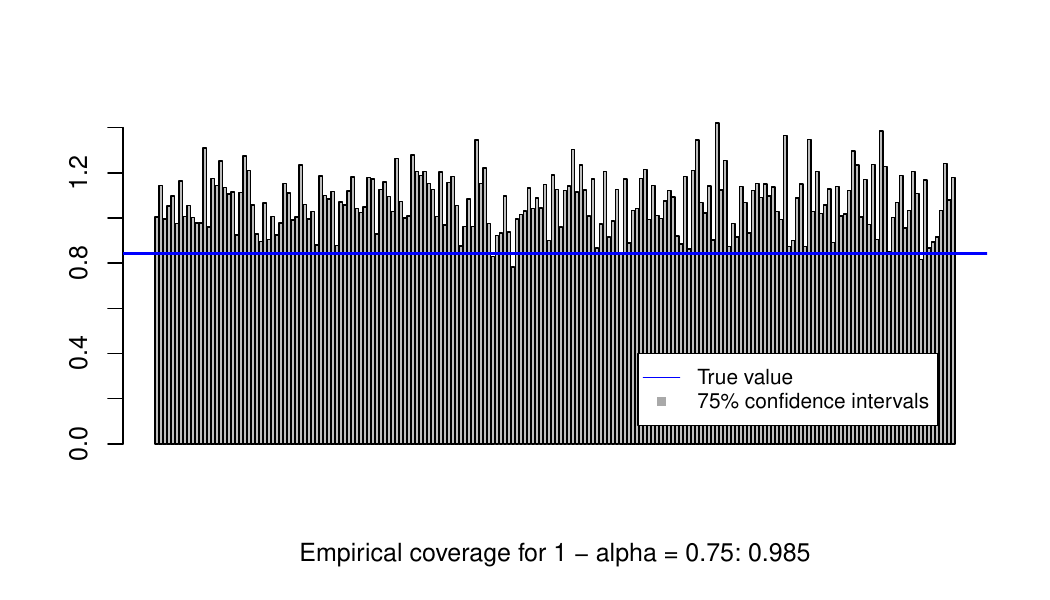}
    \caption{Empirical coverage of the $75\%$ confidence interval on $\Hnorm{f}$ for the Paley-Wiener RKHS $\cH$ and the test function $f$ of Equation \eqref{eq:PWtestFun}.}
    \label{fig:PaleyWienerNormCoverage}
\end{figure}

In a second step, we illustrate in Figure~\ref{fig:PaleyWienerRegion} the confidence region for $f$ at level $\alpha$, as derived in Corollary~\ref{cor:palwie}, for $\alpha = 0.1$ and $\alpha = 0.5$, and two different samples $\doe$.
The size of the region appears reasonable, bearing in mind that all values of $f(x)$ for $x \in [0, 1]$ are expected to lie within the region with probability at least $1 - \alpha$.
However, recall that this region is constructed from the previous confidence interval on $\Hnorm{f}$ via Proposition~\ref{prop:approxErrorRKHS}.
Based on the results above, we expect it to be conservative. The overestimation may also be amplified by the fact that the equality in Proposition~\ref{prop:approxErrorRKHS} holds for a fixed $x \in \domain$, rather than uniformly for all $x \in \domain$.
We can estimate the actual coverage of the confidence region by estimating
$$ \beta = \prob \left( \max_{x \in \domain} \, 
 g(x, \doe)  < 0 
\right)$$
where $g(x, \doe) = \vert f(x) - \hat{f}(x) \vert  - z_\alpha \sqrt{C(x)}$. In practice, we prefer considering $g(x, \doe) = (f(x) - \hat{f}(x))^2 - z_\alpha^2 C(x)$, which is a smoother function, easier to maximize. We know that $\beta \geq 1 - \alpha$. The empirical estimate of $\beta$, computed on $100$ independent samples $X_1, \dots, X_n$, is nearly equal to $1$, even for $\alpha = 0.5$. 

\begin{figure}[h!]
    \centering
    \includegraphics[width=0.48\linewidth, trim=1cm 1.5cm 1cm 1.5cm, clip = TRUE]{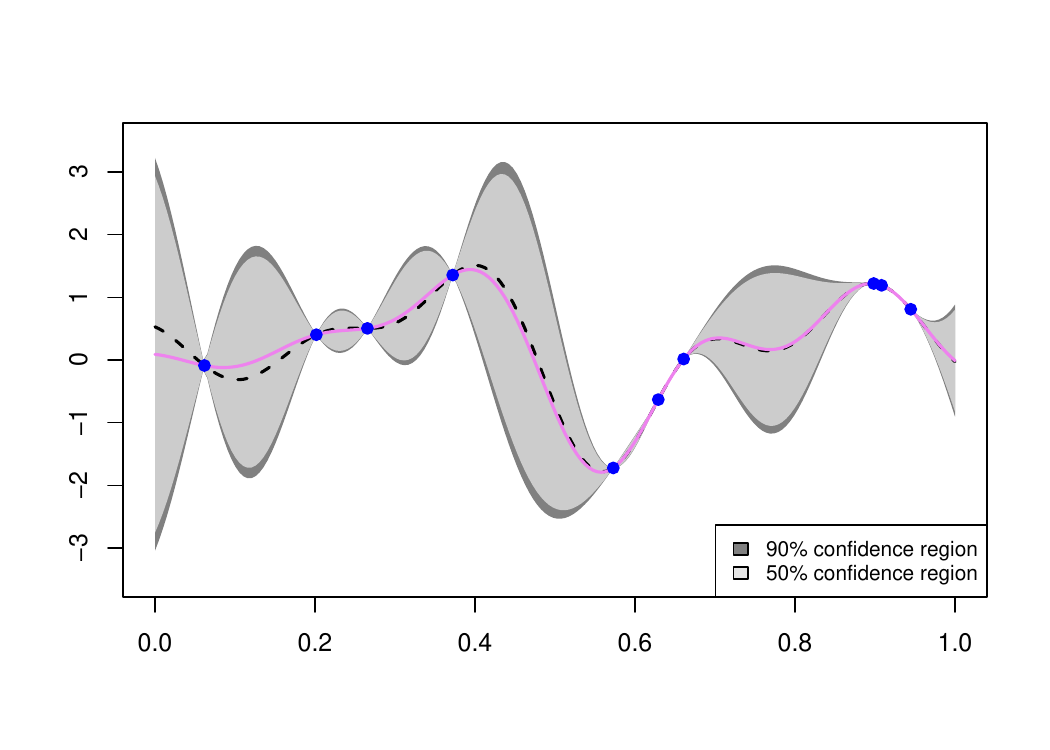} \hfill
    \includegraphics[width=0.48\linewidth, trim=1cm 1.5cm 1cm 1.5cm, clip=TRUE]{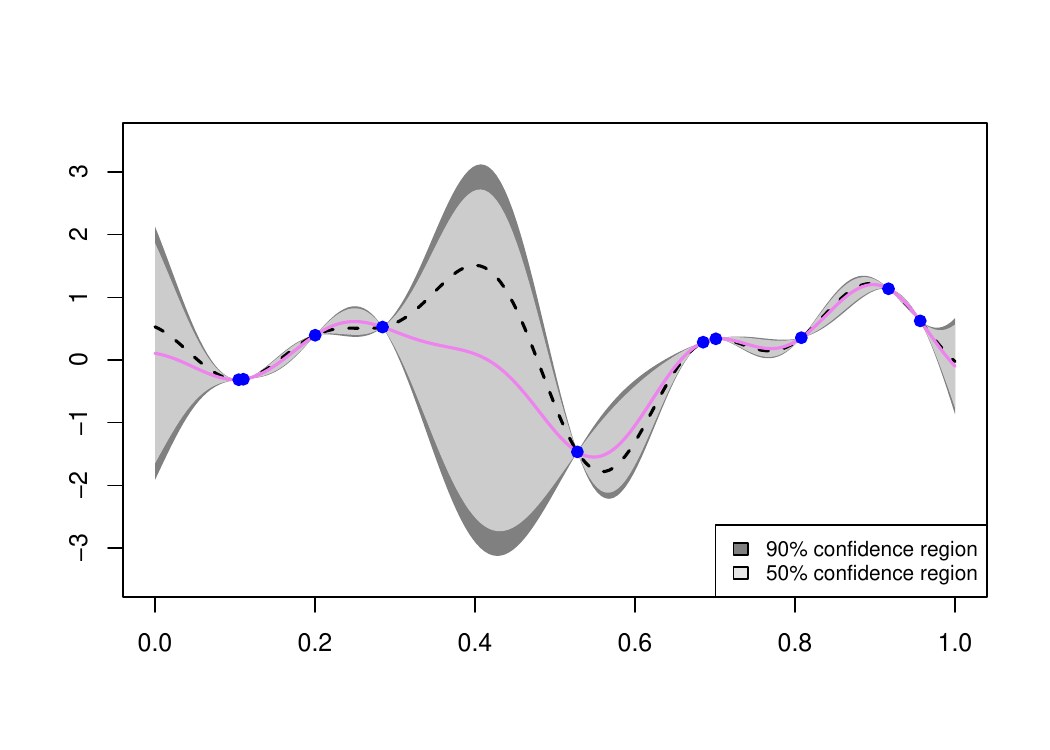}
    \caption{Non-asymptotic confidence regions in the Paley-Wiener RKHS, for two different samples of design points. Black dashed line: the unknown function $f$. Solid violet line: the estimator $\hat{f}$ constructed from $n=10$ observations of $f$ (blue points). Shaded area: $90\%$ and $50\%$ confidence regions on $f$ based on these observations.}
    \label{fig:PaleyWienerRegion}
\end{figure}

\subsection{Derivative-based confidence regions in $\HoneStandard$}
We consider the setting of Section~\ref{sect:grad}. For simplicity, we choose $\mu$ as the uniform distribution on $[0, 1]$.
Recall that the unknown function $f$ is assumed to belong to $\Htwo$ in order to ensure that its derivative is defined pointwise. To this end, we use the Poincaré basis which is an orthogonal basis of $\Hone$ whose elements lie in $\Htwo$. We then define $f$ as a finite expansion in this basis:
\begin{equation} \label{eq:HoneTestFun}
    f(x) = \sum_{\ell = 0}^{L-1} w_\ell \, \eigfunell(x)
\end{equation}
We fix $L = 20$ and we choose $w_\ell = (\ell+1)^{-2} \, \varepsilon_{\ell}$ where $\varepsilon_0, \dots, \varepsilon_{L-1}$ are sampled independently from the standard Normal distribution $\mathcal{N}(0, 1)$. These choices are motivated by the desire to capture a broad spectrum of patterns in the behavior of $f$, while dampening the influence of high‑frequency components.
Then, we set $n=10$, and sample $X_1, \dots, X_n$ independently from the uniform distribution on $[0, 1]$.

The presentation now follows the same structure as in the previous section.
As a first step, we compute the confidence interval $I_\alpha(\doe) = [0, z_\alpha]$ for the RKHS norm of $f$, as described in Proposition~\ref{prop:ConfRegionWithDer}.
Here, the value of $b$ is chosen as the exact maximum of $f^2 + f'^2$ on the interval $[0, 1]$, computed numerically.
In practice, one should instead rely on an upper bound informed by prior knowledge.
To assess the utility of the confidence interval, we first check numerically that 
$z_\alpha < \sqrt{b}$, which is expected since $b$ is an obvious upper bound for $\HoneNorm{f}^2$.
Now, we further assess the real coverage of $I_\alpha(\doe)$. In this example, thanks to \eqref{eq:eigenfunHoneNorm}, the RKHS norm of $f$ is given explicitly by
$$ \HoneNorm{f}^2 = \sum_{\ell = 0}^{L-1} (1 + \eigvalell) w_\ell^2. $$
Figure~\ref{fig:HoneNormCoverage} displays a barplot of $I_\alpha(\doe)$ for $200$ independent samples $(X_1, \dots, X_n)$ with $\alpha = 0.25$. We can see that the intervals are too conservative, with an empirical coverage of 1. Since $b$ was set to its minimal possible value, this means that the overestimation here comes from the Hoeffding's inequality used to compute $I_\alpha(\doe)$.

\begin{figure}[h!]
    \centering
    \includegraphics[width=0.8\linewidth, height=6cm, trim=1cm 1.5cm 1cm 1.5cm, clip = TRUE]{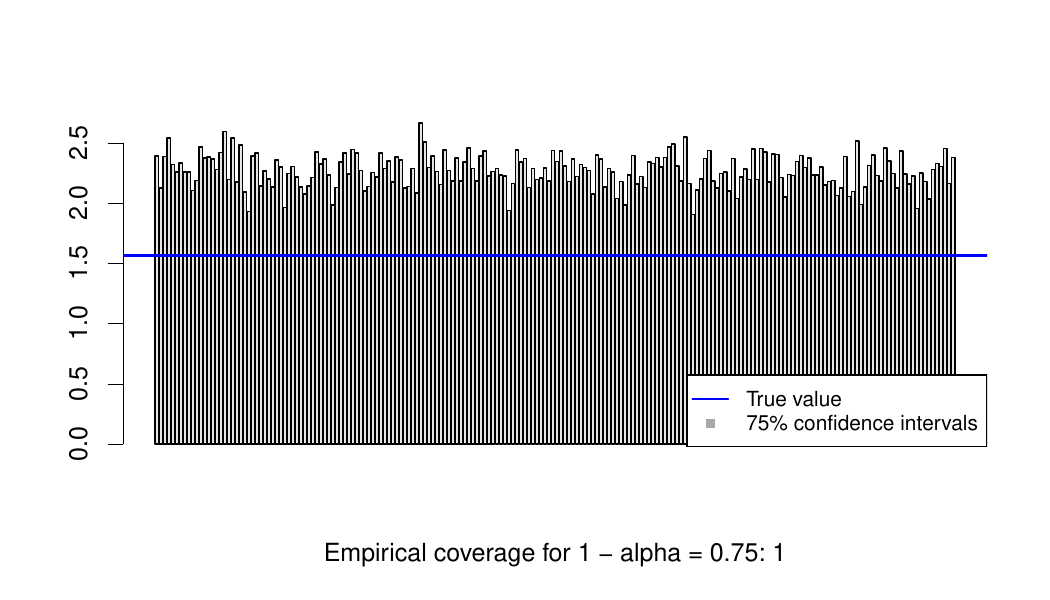}
    \caption{Empirical coverage of the $75\%$ confidence interval on $\HoneNorm{f}$ for the standard Sobolev RKHS and the test function $f$ of Equation \eqref{eq:HoneTestFun}.}
    \label{fig:HoneNormCoverage}
\end{figure}

In a second step, we illustrate in Figure~\ref{fig:HoneRegionDer} the confidence region for $f$ at level $\alpha$, as derived from Proposition~\ref{prop:ConfRegionWithDer}, for $\alpha = 0.1$ and $\alpha = 0.5$, and two different samples $\doe$.
The conclusion is similar to that of the previous section: although the region is conservative -- since $I_\alpha(\doe)$ has been shown to overestimate $\Hnorm{f}$, its overall size appears reasonable.

\begin{figure}[h!]
    \centering
    \includegraphics[width=0.48\linewidth, trim=1cm 1.5cm 1cm 1.5cm, clip = TRUE]{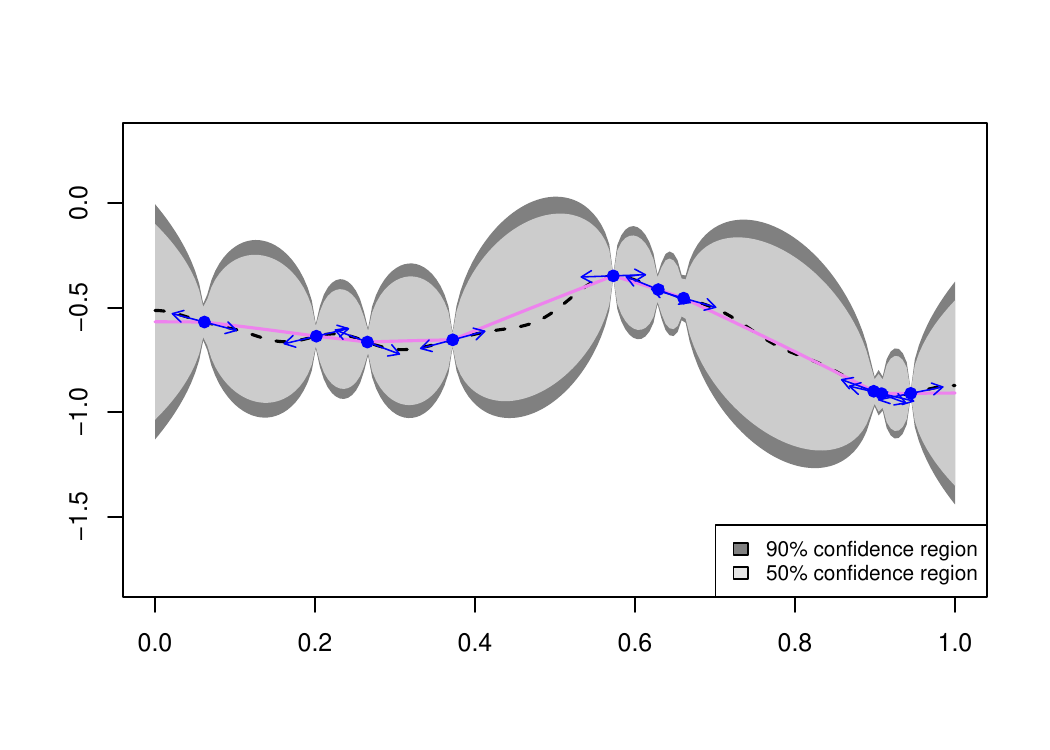} \hfill
    \includegraphics[width=0.48\linewidth, trim=1cm 1.5cm 1cm 1.5cm, clip = TRUE]{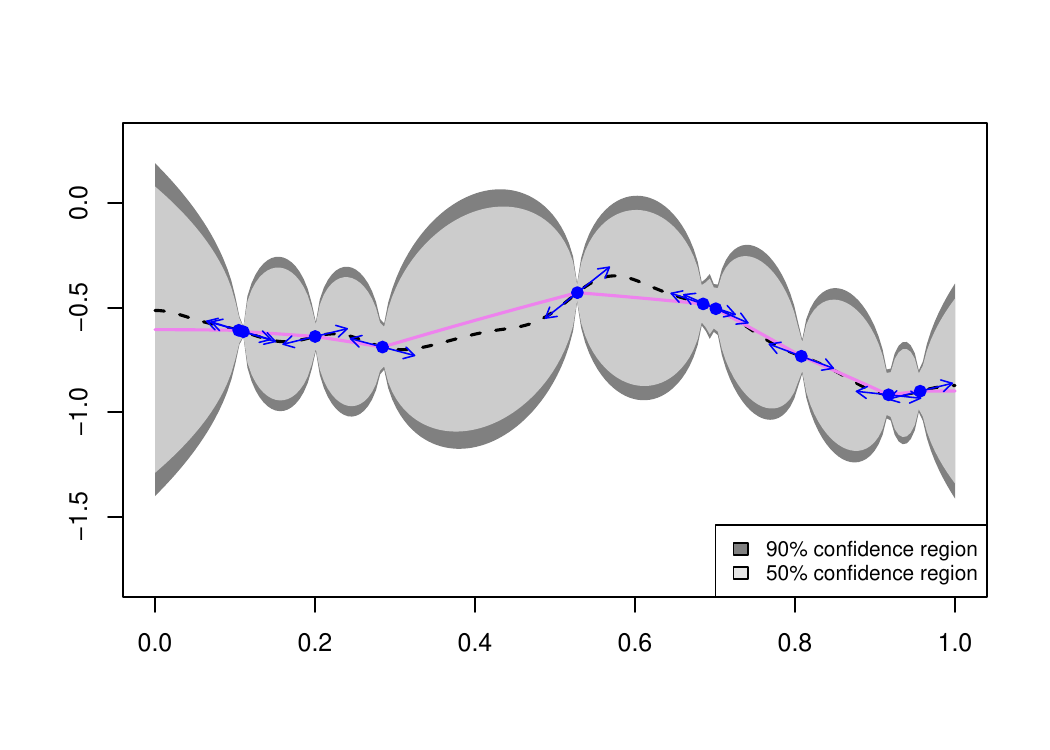}
    \caption{Non-asymptotic confidence region in $\Hone$ for two different samples of design points. Black dashed line: the unknown function $f$. Solid violet line: the estimator $\hat{f}$ constructed from $n=10$ observations of $f$ (blue points). Shaded area: $90\%$ and $50\%$ confidence regions on $f$ based on these observations and its derivatives at the same points (indicated by arrows).}
    \label{fig:HoneRegionDer}
\end{figure}

\subsection{Derivative-free confidence regions in $\HoneStandard$}

As in the previous section, we work within the standard Sobolev space $\HoneStandard$, using the same toy function \eqref{eq:HoneTestFun}. However, we now consider a setting where derivatives are no longer available and follow the framework of Section~\ref{sect:derFreeConfReg}, in which the RKHS norm is estimated using the Poincaré basis.
Since the agnostic method is not suitable for small sample sizes, we focus on scenarios where some regularity of $f$ is assumed, as detailed in Section~\ref{sect:derFreeConfRegCoefDecrease}. 

Specifically, we assume that the coefficients $\coefell$ of $f$ in the Poincaré basis decay at the rate $\eigvalell^{-p}$ and satisfy $\coefell \leq A \eigvalell^{-p}$ (see Equation~\ref{eq:coefSpeedHyp}).
We set $p = 5/2$ and choose $A = \left(\sqrt{2} \, \zeta(2p)/\pi^{2p} \right)^{-1}$ so that $\InfNorm{f} \leq 1$ (see Equation~\ref{eq:fInfUBreg}). Using Equation~\eqref{eq:fDerInfUBreg}, this further implies that $\InfNorm{f'} \leq \frac{1}{\pi} \frac{\zeta(2p-1)}{\zeta(2p)} \approx 3.28$.
For simplicity, we assume $c_0 = 0$, meaning that $f$ has been centered during a preprocessing step.

Recall that in this spectral approach, the norm is estimated using the first terms of the Poincaré basis expansion, while the remainder -- interpreted as a bias term -- is controlled through an explicit assumption (see Equation~\ref{eq:remainderAssumption}).
Here, we set $\remainderfun = \beta (N+1)^{-q}$, with $\beta = 2^{q-1}$ chosen so that $\remainderfun(1) = 0.5$. We will report the value of $\remainderfunComp_N = 1 - \remainderfun(N)$ (see Equation~\ref{eq:remainderfunComp}), as it appears directly in the expression of the confidence bound.\\

The results are summarized in Table~\ref{tab:derFreeRegion}, which displays the upper bound on $\HoneNorm{f}$ at confidence level $1 - \alpha$, as derived from Proposition~\ref{prop:derFreeRegion}, with $\alpha = 0.1$. We consider three sample sizes: $n = 10$, $100$, and $1000$, and examine two bias scenarios: a relatively large bias ($q = 1$) and a very small one ($q = 4$).
For each case, we also report the values of the quantities involved in the computation of the upper bound, for the first 10 values of $N$: $\partialSumHat$, $t$, $\deltaUB$, and $\remainderfunComp_N$. The user can then select the value of $N$ that yields the smallest possible upper bound $\HoneNorm{f}^\textrm{UB}$. This optimal choice is highlighted in bold in the table.

Several general conclusions can be drawn. First, for $n = 10$, the upper bound is significantly larger than the true value of $\HoneNorm{f}$ (approximately $1.56$), and is even not informative. Indeed, it exceeds the bound that can be obtained directly from the sup-norms of $f$ and its derivative:
$$ \HoneNorm{f} \leq \left( \InfNorm{f}^2 + \InfNorm{f'}^2 \right)^{1/2} \approx 3.43 $$
A similar situation occurs for $n = 100$ in the large bias case.
Second, the optimal number $N$ of eigenfunctions that minimizes the bound is very small -- typically $N = 1$ or $2$. This reflects the difficulty of handling series that do not converge absolutely, particularly due to the growth of the weights $1 + \eigvalell = O(\ell^2)$, which amplifies the effect of variability, especially when using concentration inequalities. For instance, when $n$ is fixed, the values of $\vert \partialSumHat \vert$ do not appear to stabilize as $N$ increases.

In conclusion, for this example, the non-asymptotic confidence interval for $\HoneNorm{f}$ is overly conservative and only becomes informative for $n > 100$. As a result, the corresponding confidence regions for $f$ are also excessively conservative, and we therefore do not display them.

\begin{table}[h!]
\footnotesize
\centering

\begin{minipage}[t]{0.45\textwidth}
$n=10$, $\remainderfun(N) \propto (N+1)^{-1}$ 
\vspace{0.2cm} \\
\begin{tabular}{rrrrrr}
\hline
$N$ & $\HoneNorm{f}^{\textrm{UB}}$ & $\Ustat$ & $t$ & $\deltaUB$ & $\remainderfunComp_N$ \\ \hline
\textbf{1} & \tb{7.57} & 1.18 & 24.84 & 2.64 & 0.50 \\ 
  2 & 8.45 & 3.21 & 32.81 & 11.63 & 0.67 \\ 
  3 & 10.52 & 4.53 & 46.89 & 31.59 & 0.75 \\ 
  4 & 13.23 & 1.95 & 71.08 & 66.91 & 0.80 \\ 
  5 & 16.47 & -4.44 & 108.56 & 121.96 & 0.83 \\ 
  6 & 20.36 & -8.35 & 162.36 & 201.14 & 0.86 \\ 
  7 & 24.25 & -29.58 & 235.48 & 308.83 & 0.88 \\ 
  8 & 29.14 & -25.56 & 330.92 & 449.42 & 0.89 \\ 
  9 & 34.19 & -26.66 & 451.64 & 627.30 & 0.90 \\ 
  10 & 39.57 & -23.85 & 600.65 & 846.84 & 0.91 \\ 
\hline
\end{tabular}\end{minipage}
\hfill
\begin{minipage}[t]{0.45\textwidth}
$n=10$, $\remainderfun(N) \propto (N+1)^{-4}$ 
\vspace{0.2cm} \\
\begin{tabular}{rrrrrr}
\hline
$N$ & $\HoneNorm{f}^{\textrm{UB}}$ & $\Ustat$ & $t$ & $\deltaUB$ & $\remainderfunComp_N$ \\ \hline
1 & 7.57 & 1.18 & 24.84 & 2.64 & 0.50 \\ 
\tb{  2} & \tb{7.27} & 3.21 & 32.81 & 11.63 & 0.90 \\ 
  3 & 9.26 & 4.53 & 46.89 & 31.59 & 0.97 \\ 
  4 & 11.91 & 1.95 & 71.08 & 66.91 & 0.99 \\ 
  5 & 15.08 & -4.44 & 108.56 & 121.96 & 0.99 \\ 
  6 & 18.88 & -8.35 & 162.36 & 201.14 & 1.00 \\ 
  7 & 22.71 & -29.58 & 235.48 & 308.83 & 1.00 \\ 
  8 & 27.49 & -25.56 & 330.92 & 449.42 & 1.00 \\ 
  9 & 32.45 & -26.66 & 451.64 & 627.30 & 1.00 \\ 
  10 & 37.74 & -23.85 & 600.65 & 846.84 & 1.00 \\ 
   \hline
\end{tabular}
\end{minipage}

\vspace{0.5cm}

\begin{minipage}[t]{0.45\textwidth}
$n=100$, $\remainderfun(N) \propto (N+1)^{-1}$ 
\vspace{0.2cm} \\
\begin{tabular}{rrrrrr}
\hline
$N$ & $\HoneNorm{f}^{\textrm{UB}}$ & $\Ustat$ & $t$ & $\deltaUB$ & $\remainderfunComp_N$ \\ \hline
1 & 3.86 & 0.42 & 6.78 & 0.24 & 0.50 \\ 
 \tb{ 2} & \tb{3.66} & 0.36 & 7.49 & 1.06 & 0.67 \\ 
  3 & 3.90 & 0.48 & 8.03 & 2.87 & 0.75 \\ 
  4 & 4.33 & 0.14 & 8.79 & 6.08 & 0.80 \\ 
  5 & 5.08 & 0.51 & 9.89 & 11.09 & 0.83 \\ 
  6 & 6.19 & 3.05 & 11.46 & 18.29 & 0.86 \\ 
  7 & 7.02 & 1.47 & 13.57 & 28.08 & 0.88 \\ 
  8 & 8.21 & 2.72 & 16.32 & 40.86 & 0.89 \\ 
  9 & 9.52 & 4.81 & 19.80 & 57.03 & 0.90 \\ 
  10 & 10.80 & 4.90 & 24.08 & 76.99 & 0.91 \\ 
  \hline
\end{tabular}\end{minipage}
\hfill
\begin{minipage}[t]{0.45\textwidth}
$n=100$, $\remainderfun(N) \propto (N+1)^{-4}$ 
\vspace{0.2cm} \\
\begin{tabular}{rrrrrr}
\hline
$N$ & $\HoneNorm{f}^{\textrm{UB}}$ & $\Ustat$ & $t$ & $\deltaUB$ & $\remainderfunComp_N$ \\ \hline
1 & 3.86 & 0.42 & 6.78 & 0.24 & 0.50 \\ 
 \tb{ 2} & \tb{3.14} & 0.36 & 7.49 & 1.06 & 0.90 \\ 
  3 & 3.43 & 0.48 & 8.03 & 2.87 & 0.97 \\ 
  4 & 3.90 & 0.14 & 8.79 & 6.08 & 0.99 \\ 
  5 & 4.65 & 0.51 & 9.89 & 11.09 & 0.99 \\ 
  6 & 5.74 & 3.05 & 11.46 & 18.29 & 1.00 \\ 
  7 & 6.57 & 1.47 & 13.57 & 28.08 & 1.00 \\ 
  8 & 7.74 & 2.72 & 16.32 & 40.86 & 1.00 \\ 
  9 & 9.04 & 4.81 & 19.80 & 57.03 & 1.00 \\ 
  10 & 10.30 & 4.90 & 24.08 & 76.99 & 1.00 \\ 
   \hline
\end{tabular}
\end{minipage}

\vspace{0.5cm}

\begin{minipage}[t]{0.45\textwidth}
$n=1\,000$, $\remainderfun(N) \propto (N+1)^{-1}$ 
\vspace{0.2cm} \\
\begin{tabular}{rrrrrr}
\hline
$N$ & $\HoneNorm{f}^{\textrm{UB}}$ & $\Ustat$ & $t$ & $\deltaUB$ & $\remainderfunComp_N$ \\ \hline
1 & 2.27 & 0.44 & 2.11 & 0.02 & 0.50 \\ 
 \tb{ 2} & \tb{2.22} & 0.88 & 2.29 & 0.10 & 0.67 \\ 
  3 & 2.54 & 2.22 & 2.35 & 0.28 & 0.75 \\ 
  4 & 2.55 & 2.19 & 2.39 & 0.60 & 0.80 \\ 
  5 & 2.68 & 2.45 & 2.43 & 1.10 & 0.83 \\ 
  6 & 2.77 & 2.30 & 2.49 & 1.81 & 0.86 \\ 
  7 & 2.94 & 2.21 & 2.56 & 2.78 & 0.88 \\ 
  8 & 3.14 & 2.06 & 2.65 & 4.05 & 0.89 \\ 
  9 & 3.35 & 1.71 & 2.76 & 5.65 & 0.90 \\ 
  10 & 3.65 & 1.60 & 2.89 & 7.63 & 0.91 \\ 
   \hline
\end{tabular}\end{minipage}
\hfill
\begin{minipage}[t]{0.45\textwidth}
$n=1\,000$, $\remainderfun(N) \propto (N+1)^{-4}$ 
\vspace{0.2cm} \\
\begin{tabular}{rrrrrr}
\hline
$N$ & $\HoneNorm{f}^{\textrm{UB}}$ & $\Ustat$ & $t$ & $\deltaUB$ & $\remainderfunComp_N$ \\ \hline
1 & 2.27 & 0.44 & 2.11 & 0.02 & 0.50 \\ 
 \tb{ 2} & \tb{1.91} & 0.88 & 2.29 & 0.10 & 0.90 \\ 
  3 & 2.24 & 2.22 & 2.35 & 0.28 & 0.97 \\ 
  4 & 2.29 & 2.19 & 2.39 & 0.60 & 0.99 \\ 
  5 & 2.45 & 2.45 & 2.43 & 1.10 & 0.99 \\ 
  6 & 2.57 & 2.30 & 2.49 & 1.81 & 1.00 \\ 
  7 & 2.75 & 2.21 & 2.56 & 2.78 & 1.00 \\ 
  8 & 2.96 & 2.06 & 2.65 & 4.05 & 1.00 \\ 
  9 & 3.18 & 1.71 & 2.76 & 5.65 & 1.00 \\ 
  10 & 3.48 & 1.60 & 2.89 & 7.63 & 1.00 \\ 
   \hline
\end{tabular}
\end{minipage}
\caption{Upper bound of $\HoneNorm{f} \approx 1.56$ with probability greather than $0.9$, for various values of sample size $n$, number of eigenfunctions $N$ used in the spectral estimator, and two scenarii for the bias term $\remainderfun(N)$.}
\label{tab:derFreeRegion}
\end{table}

\section{Conclusion and perspectives}
\label{sect:concl}
In our paper, we address the challenging problem of constructing non-asymptotic global confidence regions for functions belonging to a RKHS observed at i.i.d. random points. This ambitious objective significantly surpasses the scope of classical conformal prediction, offering a means to construct multiple non-asymptotic conditional prediction intervals.
\ \\

\noindent
Traditionally, this problem is approached within the Bayesian paradigm, particularly in the context of Gaussian processes. However, the Bayesian approach presents two key limitations. First, it relies heavily on strong modeling assumptions. Second, it typically yields only marginal confidence intervals rather than global ones.
In contrast, our work demonstrates that the construction of a global non-asymptotic confidence region becomes feasible if one can estimate the RKHS norm of the unknown function and has access to non-asymptotic bounds on the deviation of the estimator. To this end, we use classical Hoeffding's concentration inequalities.
\ \\

\noindent
We provide a complete treatment of the Paley–Wiener space, correcting some imprecise results from the existing literature. The  Sobolev space $\cH^1$ is also considered. In this case, we show that Hoeffding's inequalities again yields reasonable confidence regions, provided that the derivative is observed alongside function values. When the derivative information is not available we propose a spectral approach to construct confidence regions. However, this method tends to be overly conservative.
A notable empirical trend is the systematic overestimation of the target function by the resulting confidence regions. This overconservativeness can be attributed to two main factors. First, the non-asymptotic confidence bounds are derived using Hoeffding's inequality.  This inequality is known to yield conservative estimates because of its generality and lack of distributional assumptions. Second, the local approximation error bound (\ref{eq:ApproxErrorIneq}) is established via the Cauchy–Schwarz inequality.  While analytically convenient, it introduces additional looseness in the resulting bound.
\ \\

\noindent
In future works, extending our analysis to general RKHS settings is a major challenge. It lies in estimating the norm of the unknown function. This may potentially be addressed using a regularized version of the plug-in estimator based on the inverse of the kernel Gram matrix. Nevertheless, obtaining non-asymptotic deviation bounds for such estimators appears to be a significantly difficult task. In practice, pursuing asymptotic guarantees via a central limit theorem may offer a more tractable and realistic approach.

\subsection*{Appendix A. Hoeffding's inequalities}

We recall below two Hoeffding's inequalities, as they are stated in \cite{Hoeffding_1963}, for upper tails. 
Inequalities for lower tails are obviously deduced, in Prop.~\ref{prop:HoeffdingClassical} by changing the signs of $X_i, a_i, b_i$ and replacing $\prob \left( \overline{X} - \Esp(X) > t \right)$ by $\prob \left( \overline{X} - \Esp(X) < - t \right)$, and in Prop.~\ref{prop:HoeffdingUstat} by changing the signs of $U, a, b$ and replacing $\prob \left( U - \Esp(U) > t \right)$ by $\prob \left( U - \Esp(U) < - t \right)$.\\

\begin{prop}[Classical Hoeffding's inequality \cite{Hoeffding_1963}, Theorem 2]
\label{prop:HoeffdingClassical}
Let $X_1, \dots, X_n$ be independent variables verifying $a_i < X_i < b_i$ for all $i=1, \dots, n$. Denote $\overline{X} = \frac{X_1+ \dots + X_n}{n} $. Then, for all $t>0$, 
$$ \prob \left( \overline{X} - \Esp(X) > t \right) < \exp \left( - \frac{2n^2 t^2}{\sum_{i=1}^n (b_i - a_i)^2} \right)  $$
\end{prop}

\begin{prop}[Hoeffding's inequality for U-statistics \cite{Hoeffding_1963}, \S 5a, case $r=2$]
\label{prop:HoeffdingUstat}
Consider the U-statistics
$U = \frac{2}{n(n-1)} \sum_{i<j} g(X_i, X_j)$ where $X_1, \dots, X_n$ are independent random variables and $g$ is a bounded function: $a < g(x,y) < b$ (for all $x, y$). Then for all $t>0$,
$$ \prob \left( U - \Esp(U) > t \right) < \exp \left( - \frac{2k t^2}{(b - a)^2} \right)  $$
where $k = [n/2]$.
\end{prop}

\subsection*{Appendix B. Upper bounds of basis coefficients and related quantities.}

Let $f \in \Hone$. We consider the expansion of $f$ in the Poincaré basis, viewed as an orthornormal basis of $\Ltwo$:
$$ f = \sum_{\ell \geq 0} \coefell \eigfunell $$
with
\begin{equation} \label{eq:coefell}
\coefell = \langle f, \eigfunell \rangle
\end{equation}
Furthermore, for $\ell \in \N$, let us define the function
\begin{equation} \label{eq:hell}
    \hell = f \eigfunell 
\end{equation}
We now derive upper bounds of $\coefell$, $\InfNorm{\hell - \coefell}$ and $\Var_\mu \hell $.
For simplicity, we assume in all the section that $\ell \geq 1$. 

\subsubsection*{Bounding $\coefell$}
Using the properties of the Poincaré basis, we have
$$ \coefell 
= \langle f, \eigfunell \rangle 
= \frac{1}{\sqrt{\eigvalell}} \langle f', \frac{1}{\sqrt{\eigvalell}}\eigfunell' \rangle 
= - \frac{1}{\eigvalell}
\left( [f' \eigfunell]_0^1 -  \langle f'', \eigfunell  \rangle \right)
$$
The second equality comes from \eqref{eq:PoincCaracHone}. The third one is obtained by integration by part when it is valid, typically if $f \in \Htwo$.
Now, using the Cauchy-Schwartz inequality, we deduce the three inequalities
\begin{eqnarray}
\vert \coefell \vert &\leq& \LtwoNorm{f} \\
\vert \coefell \vert &\leq& \frac{1}{\sqrt{\eigvalell}} \LtwoNorm{f'} \\
\vert \coefell \vert &\leq& \frac{1}{\eigvalell} 
\left(
\sum_{i=0, 1}
 \vert f'(i) \eigfunell(i) \vert + 
\LtwoNorm{f''} 
\right)
\end{eqnarray}
The first two inequalities are sharp, as they are equalities if $f = \eigfunell$.
However, regarding the dependence with $\ell$, the first one is too loose as it does not tend to zero with $\ell$.
The last one is the only one that gives a finite upper bound of $\HoneNorm{f}$ 
$$ \HoneNorm{f}^2 = \sum_\ell (1+ \eigvalell) \coefell^2 \leq c_0^2 + \sum_{\ell \geq 1} \frac{1+\eigvalell}{\eigvalell^2} \left(
\sqrt{2} (
\vert f'(0) \vert + \vert f'(1) \vert ) + \LtwoNorm{f''} 
\right)^2
< + \infty$$
Note however that it requires more regularity on $f$.\\

Finally, the upper bounds above implies the upper bound using the infinite norm of $f$ and its derivative,
\begin{equation}
\vert \coefell \vert 
\leq 
\coefellUB(f) := \min \left(
\InfNorm{f}, 
\frac{\InfNorm{f'} 
}{\sqrt{\eigvalell}} \right)
\label{eq:coefUBWithDer} 
\end{equation}
Here $\eigfunell$ are bounded uniformly with respect to $\ell$, we can do better by further assuming more regularity on $f$. Then, we will have
\begin{equation}
\vert \coefell \vert \leq \min \left(\InfNorm{f}, \frac{\InfNorm{f'}}{\sqrt{\eigvalell}} ,
\frac{ 2 \sqrt{2} \InfNorm{f'} + \InfNorm{f''}}{\eigvalell}  \right)
\label{eq:coefBoundWithDertwo}
\end{equation}

\subsubsection*{Bounding $\InfNorm{\hell - \coefell}$}
Let us consider the function $\helltilde$ obtained by centering $\hell$ (with respect to $\mu$), 
\begin{equation} \label{eq:helltilde}
    \helltilde = \hell - \coefell
\end{equation}
A straightforward upper bound of the infinite norm of $\helltilde$ is obtained from
\eqref{eq:coefUBWithDer}:
\begin{equation} \label{eq:helltildeUBdirect}
   \InfNorm{\helltilde} \leq \InfNorm{f} \InfNorm{\eigfunell} 
+ \coefellUB(f) 
\end{equation}
Alternatively, as  
$ \helltilde$ has at least one zero in $[0, 1]$,
we obtain a Lipschitz upper bound:
\begin{equation} \label{eq:helltildeUBlip}
 \InfNorm{\helltilde} 
\leq \InfNorm{\hell'} 
\leq \InfNorm{f} \eigenfunDerInf + \InfNorm{f'} \eigenfunInf  
\end{equation}
These two inequalities are complementary. \eqref{eq:helltildeUBlip} uses the derivative of $f$, but the corresponding upper bound 
tends to infinity with $\ell$ while the upper bound of \eqref{eq:helltildeUBdirect} is bounded by $\InfNorm{f}( \sqrt{2} +1)$. Finally, we can use the minimum:
\begin{equation} \label{eq:helltildeUB}
\InfNorm{\helltilde} 
\leq 
\helltildeUB(f) :=
\min \left(
\InfNorm{f} \InfNorm{\eigfunell} 
+ \coefellUB(f),
\InfNorm{f} \eigenfunDerInf + \InfNorm{f'} \eigenfunInf 
\right)
\end{equation}

\subsubsection*{Bounding $\Var_\mu \hell $}
A sharp inequality of $\Var_\mu \hell$, with equality case for constant functions, is simply 
\begin{equation} \label{eq:varHellUB}
    \Var_\mu \hell 
    \leq  \InfNorm{f}^2.
\end{equation}
It is obtained by writing
$$ \Var_\mu \hell
\leq \Esp_\mu \left( h_\ell^2 \right )
\leq \InfNorm{f}^2 \Esp_\mu \left(\eigfunell^2 \right) 
= \InfNorm{f}^2.$$

\subsection*{Acknowledgments}
The authors would like to thank the Isaac Newton Institute for Mathematical Sciences, Cambridge, for support and hospitality during the programme ``Calibrating prediction uncertainty: statistics and machine learning perspectives'' where work on this paper was undertaken. This work was supported by EPSRC grant no EP/R014604/1. 
Our work has benefitted from the AI Interdisciplinary Institute ANITI. ANITI is funded by the France 2030 program under the Grant agreement ANR-23-IACL-0002.

\bibliographystyle{abbrv}
\bibliography{bibi}

\end{document}